\documentclass[12pt,a4paper]{amsart}
\usepackage[top=35mm, bottom=30mm, left=30mm, right=30mm]{geometry}
\usepackage{mathptmx}
\usepackage{mathrsfs}

\usepackage{verbatim}
\usepackage{url}
\usepackage[all]{xy}
\usepackage{xcolor}
\usepackage[colorlinks=true,citecolor=blue]{hyperref}

\usepackage{amsmath,amsthm}
\usepackage{amssymb}

\usepackage{enumerate}

\usepackage{graphicx}

\newtheorem{thm}{Theorem}[section]

\theoremstyle{definition}

\numberwithin{equation}{section}

\begin{document}


\baselineskip=17pt


\title{Distal higher rank lattice actions on  surfaces}

\author{Enhui Shi}

\address[E.H. Shi]{School of Mathematical Sciences, Soochow University, Suzhou 215006, P. R. China}
\email{ehshi@suda.edu.cn}

\begin{abstract}
Let $\Gamma$ be a lattice in ${\rm SL}(n, \mathbb R)$ with $n\geq 3$ and $\mathcal S$ be a closed surface. Then
$\Gamma$ has no distal minimal action on $\mathcal S$.
\end{abstract}


\maketitle

\pagestyle{myheadings} \markboth{E. H. Shi }{The realization and classification}

\section{Introduction}

The notion of distality was introduced by Hilbert for better understanding equicontinuity (\cite{El}). The study of minimal distal systems culminates in the
 beautiful structure theorem of H. Furstenberg (\cite{Fu}), which describes completely the relations between distality and equicontinuity for minimal systems.
 An interesting question is what compact manifold can support a  distal minimal group action? Clearly, the answer to this question depends on the
 topology of the phase space and the algebraic structure of the acting group. A remarkable result says that if a nontrivial space $X$ admits a
 distal minimal actions by abelian groups, then $X$ cannot be simply connected (see e.g. \cite[Chapter 7-Theorem 16]{Au}). Thus the $2$-sphere $\mathbb S^2$
does not admit any distal minimal abelain group actions. However, it is very easy to construct a minimal isometric nonabelian group action on $\mathbb S^2$ generated by two
rigid irrational rotations around different coordinate axes of $\mathbb R^3$.

We consider the distal minimal  actions of some higher rank lattices on closed surfaces ($2$-dimensional compact connected manifolds) and obtained
the following result.

\begin{thm}\label{main theorem}
Let $\Gamma$ be a lattice in ${\rm SL}(n, \mathbb R)$ with $n\geq 3$ and $\mathcal S$ be a closed surface. Then
$\Gamma$ has no distal minimal action on $\mathcal S$.
\end{thm}

Here we remark that this theorem does not hold for the case $n=2$, since the free nonabelian group $\mathbb Z*\mathbb Z$
is a lattice in ${\rm SL}(2, \mathbb R)$ and we have mentioned that  $\mathbb Z*\mathbb Z$ can act on $\mathbb S^2$ minimally
and distally. In addition, there do exist isometric minimal actions of some lattices in ${\rm SL}(n, \mathbb R)$ on compact manifolds
with dimension $>2$ (see e.g. \cite{Mo} or \cite[p.59]{Br}). We should notice that there are nondistal minimal higher rank actions
on surfaces, such as the action of ${\rm SL}(3, \mathbb Z)$ on $\mathbb {RP}^2$ induced by the linear action on $\mathbb R^3$.

\section{Preliminaries}

In this section, we will recall some basic notions and introduce some results which will be used in the proof of the main theorem.

\subsection{Distal actions}

Let $X$ be a topological space and let ${\rm Homeo}(X)$ be
the homeomorphism group of $X$. Suppose $G$ is a topological group. A group
homomorphism $\phi: G\rightarrow {\rm Homeo}(X)$ is called a {\it continuous
action} of $G$ on $X$ if $(x, g)\mapsto \phi(g)(x)$ is continuous; we use the symbol $(X, G, \phi)$ to denote this action.
The action $\phi$ is said to be {\it faithful} if it is injective.  For brevity, we usually use $gx$ or $g(x)$ instead of $\phi(g)(x)$
and use $(X, G)$ instead of $(X, G, \phi)$ if no confusion occurs.

For $x\in X$,  the {\it orbit} of $x$  is the set $Gx\equiv\{gx:g\in
G\}$; $K\subset X$ is called {\it $G$ invariant} if $Gx\in K $ for every $x\in X$;
$(X, G, \phi)$ is called {\it minimal} if $Gx$ is dense in $X$ for every $x\in X$, which is equivalent
to that $G$ has no proper closed invariant set; is called {\it transitive} if $Gx=X$ for every $x\in X$.
If $K\subset X$ is $G$ invariant, then we naturally get a restriction action $\phi|K$ of $G$ on $X$;
if $K$ is closed and nonempty, and the restriction action $(K, G, \phi|K)$ is minimal,
then we call $K$ a {\it minimal set}  of $X$ or of the action. It is well known that $(X, G, \phi)$
always has a minimal set when $X$ is a compact metric space.

Suppose $(X, G, \phi)$ and  $(Y, G, \psi)$ are two actions. If there is a continuous surjection  $f:X\rightarrow Y$ such that
$f(\phi(g)x)=\psi(g)f(x)$ for every $g\in G$ and every $x\in X$, then we say $f$ is a {\it homomorphism} and $(Y, G, \psi)$
is a {\it factor} of $(X, G, \phi)$. If $Y$ is a single point, then we call $(Y, G, \psi)$  a {\it trivial factor} of  $(X, G, \phi)$.

Assume further that $X$ is a compact metric space with metric $d$. The action $(X, G, \phi)$ is called
{\it equicontinuous} if for every $\epsilon>0$ there is a $\delta>0$ such that $d(gx, gy)<\epsilon$ whenever $d(x, y)<\delta$;
is called {\it distal}, if for every $x\not=y\in X$, $\inf_{g\in G}d(gx, gy)>0$. Clearly, equicontinuity implies distality.

The following results can be found in \cite{Au}.


\begin{thm}[\cite{Au}, p.98]\label{homomorphism open}
Let $(X, G, \phi)$ and  $(Y, G, \psi)$ be distal minimal actions, and let $f:X\rightarrow Y$ be a homomorphism.
Then $f$ is open.
\end{thm}

\begin{thm}[\cite{Au}, p.104]\label{maximal fator}
Let $(X, G, \phi)$ be a distal minimal action. If $X$ is not a single point, then $(X, G, \phi)$ has a nontrivial equicontinuous factor.
\end{thm}

\begin{thm}[\cite{Au}, p.52]\label{compact group}
Let $(X, G, \phi)$ be  equicontinuous. Then the closure $\overline{\phi(G)}$ in ${\rm Homeo}(X)$ with respect to
the uniform convergence topology is a compact topological group.
\end{thm}

\subsection{Compact transformation groups}

Let $(X, G, \phi)$ be a group action and $H$ be a closed subgroup of $G$. Then we use $X/H$ to denote the orbit space
under the $H$ action, which is endowed with the quotient space topology. We use $G/H$ to denote the coset space with the quotient topology,
which is also the orbit space obtained by the left translation action on $G$ by $H$. If $H$ is a normal closed subgroup of $G$, then
$G/H$ is a topological group.

The following theorems can be seen in \cite{MZ}. We only state them in some special cases which are enough for our uses.

\begin{thm}[\cite{MZ}, p.65]\label{homogeneous space}
Let $X$ be a compact metric space and let $(X, G)$ be an action of group $G$ on $X$. Suppose $G$ is compact. Then for every
$x\in X$, $G/G_x$ is homeomorphic to $Gx$, where $G_x=\{g\in G: gx=x\}$.
\end{thm}

\begin{thm}[\cite{MZ}, p.99]\label{small group} Let $G$ be a compact group and let $U$ be an open neighborhood of the identity $e$.
Then $U$ contains a normal subgroup $H$ of $G$ such that $G/H$ is isomorphic to a Lie group.
\end{thm}

\begin{thm}[\cite{MZ}, p.61]\label{induce action} Let $X$ be a compact metric space and let $(X, G)$ be an action of group $G$ on $X$.
 Suppose $G$ is compact and $H$ is a closed normal subgroup of  $G$. Then  $G/H$ can
 act on $X/H$ by letting $gH.H(x)=H(gx)$ for $gH\in G/H$ and $H(x)\in X/H$.
\end{thm}

\subsection{Higher rank lattices}

A subgroup $\Gamma$ of a Lie group $G$ is called a lattice in $G$ if $\Gamma$ is discrete and $G/\Gamma$ has finite volume. If
$\Gamma$ is a lattice of $G$ and $G/\Gamma$ is compact, then $\Gamma$ is called a {\it cocompact} lattice of $G$. It is well known
that  ${\rm SL}(n, \mathbb R)$ with $n\geq 3$ always has cocompact and non-cocompact lattices;  ${\rm SL}(n, \mathbb Z)$ is
a non-cocompact lattice of  ${\rm SL}(n, \mathbb R)$. One may consult \cite{Wi1} for the examples of cocompact lattices in  ${\rm SL}(n, \mathbb R)$.





The following theorem can be deduced from \cite[Theorem VII.6.5]{Ma} and \cite[Corollary 16.4.2]{Wi1} (see also \cite[Theorem 12.4]{Br}).

\begin{thm}\label{lower dimension} Let $\Gamma$ be a lattice in ${\rm SL}(n, \mathbb R)$ with $n\geq 3$ and let $H$ be a compact Lie group.
Suppose $\phi:\Gamma\rightarrow H$ is a group homomorphism. If $\phi(\Gamma)$ is infinite, then ${\rm dim}(H)\geq n^2-1$.
\end{thm}

\subsection{Open maps on $\mathcal S$}

There are several different ways to define the dimension of a topological space. One may consult \cite{En} for
these definitions. It is well known that when $X$ is a compact metric space, then all these dimensions of $X$
coincide. In general, if $f:X\rightarrow Y$ is a continuous surjection between compact metric spaces $X$ and $Y$,
then the dimension ${\rm dim}(Y)$ can be larger than ${\rm dim}(X)$, even if $f$ is open (see \cite{Wa}).
However, the situation is completely different if $X$ is a surface.

The following theorem is implied by the work in the papers \cite{Mo}, \cite{RS}, \cite{St}, and \cite{Wh}.

\begin{thm}\label{upper dimension}
Let $\mathcal S$ be a closed surface and $M$ be a  manifold. If there is an open surjection $f:\mathcal S\rightarrow M$,
then ${\rm dim}(M)\leq 2$.
\end{thm}

\subsection{Dimension of compact subgroups of ${\rm Homeo}(\mathcal S)$} It is well known that if $M$ is a
connected compact Riemannian manifold of dimension $n$ and $I(M)$ is the isometric transformation group of $M$,
then $I(M)$ is a compact Lie group of dimensional at most $n(n+1)/2$ (see e.g. \cite{Ko}). In general, for a compact group
$G$ of ${\rm Homeo}(M)$, it no longer preserves the Riemannian metric on $M$. However, $G$ can still
preserve a compatible metric on $M$. In fact, let $\mu$ be the Haar measure of $G$ and $\rho$ be any compatible
metric on $M$. Then the metric $d$ on $M$ defined by $d(x,y)=\int_{g\in G} \rho(gx,gy)d\mu(g)$ 
is $G$ invariant.

The following theorem is \cite[Proposition 4.1]{Gh}.

\begin{thm}\label{compact unique}
Up to conjugacy, the rotation group ${\rm SO}(2, \mathbb R)$ is the only maximal compact subgroup of ${\rm Homeo}_+(\mathbb S^1)$.
\end{thm}

The following theorem is well known in the theory of topology (see e.g. \cite[Corollary 8.17]{Nad}).
\begin{thm}\label{peano}
Any continuous Hausdorff image of a locally connected continuum is locally connected.
\end{thm}

\begin{thm}\label{dimension surface}
Let $G$ be a compact Lie group acting faithfully and transitively on a closed surface $\mathcal S$. Then ${\rm dim}(G)\leq 3$.
\end{thm}

\begin{proof}
Fix a $G$ invariant metric $d$ on $\mathcal S$ and fix a point $p\in \mathcal S$. Since $G$ acts on $\mathcal S$ transitively,
$G/G_p$ is homeomorphic to $\mathcal S$ by Theorem \ref{homogeneous space}, where $G_p$ is the closed subgroup of $G$ which fix $p$.
Note that $G_p$ is also a Lie subgroup of $G$, since a closed subgroup of a Lie group is a Lie group (see \cite[Theorem 3.42]{War}).
Thus ${\rm dim}(G)={\rm dim}(G_p)+{\rm dim}(\mathcal S)={\rm dim}(G_p)+2$. We need only to show that
${\rm dim}(G_p)\leq 1$. Let $F$ be the connected component of $e$ in $G_p$. Then $F$ is a closed and hence a Lie subgroup of $G_p$ which
has the same dimension as that of $G_p$. If $F=e$, then ${\rm dim}(G_p)={\rm dim}(F)=0\leq 1$
and we are done. So, we assume that $F$ is a nontrivial connected Lie group in the sequel.

Take a closed disk $D$ in $\mathcal S$ such that the interior ${\rm Int}(D)$  contains $p$, and take an $\epsilon>0$ such
that the closed  ball $\overline{B_d(p, \epsilon)}\subset {\rm Int}(D)$. Since $F$ preserves the metric $d$
and leaves $p$ invariant, $B_d(p, \epsilon)$ is $F$ invariant.  Let $D_1$ be a closed disk such that
$p\in {\rm Int}(D_1)\subset D_1\subset B_d(p, \epsilon)$.  Let $K=\cup_{g\in F}gD_1$. Then $K$ is an $F$ invariant
locally connected continuum by \ref{peano}, which is the closure of the connected open set $\cup_{g\in F}g({\rm Int}(D_1))$.
Similar to the argument in \cite[corollary 4.7]{BK}, we then get an $F$ invariant closed disk $D_3\subset {\rm Int}(D)$
with $p\in {\rm Int}(D_3)$. Then the boundary $\partial D_3$ is an $F$ invariant simple closed curve; in fact,
$\partial D_3$ is the boundary of the (unique) unbounded component of ${\rm Int}(D)\setminus K$.

Define a map $R: F\rightarrow {\rm Homeo}(\partial D_3)$ by letting $R(g)=g|\partial D_3$.
If there is some $g\not=e\in F$ such that $g$ fixes every point of $\partial D_3$,
then $g$ fixes every point of $\mathcal S$ (see the proof of \cite[Lemma 4.8]{BK}),
which contradicts the assumption that the action of $G$ on $\mathcal S$ is faithful.
Therefore, $R$ is a continuous isomorphism between $F$ and $R(F)$; particularly, $R(F)$ is a
compact connected subgroup of ${\rm Homeo}(\partial D_3)$. Then
we get ${\rm dim}(F)={\rm dim}(R(F))=1$ by Theorem \ref{compact unique}.

All together, we have ${\rm dim}(F)\leq 1$ and hence ${\rm dim}(G)\leq 3$.
\end{proof}

\section{Proof of the main theorem}

\begin{proof}

Assume to the contrary that there is a distal minimal action $\phi:\Gamma\rightarrow {\rm Homeo}(\mathcal S)$.
  By Theorem \ref{maximal fator}, there is a nontrivial equicontinuous factor $(K, \Gamma, \psi)$
(which is also minimal). Set $H=\overline{\psi(\Gamma)}$. From Theorem \ref{compact group},  $H$ is a compact subgroup of ${\rm Homeo}(K)$.
Applying Theorem \ref{small group}, we can take a small normal subgroup $H'$ of $H$ such that $H/H'$ is a Lie group and $H'y$ is a proper
subset of $K$ for every $y\in K$. Then we get a continuous transitive
action $\psi'$ of the Lie group $H/H'$ on the quotient space $K/H'$ by Theorem \ref{induce action}; in particular, $K/H'$ is a connected compact manifold
of dimension $\geq 1$. Since $(K/H', \Gamma, \psi'\psi)$ is a nontrivial factor of $(\mathcal S, \Gamma, \phi)$, by Theorem \ref{homomorphism open}
and Theorem \ref{upper dimension}, we have ${\rm dim} (K/H')\leq 2$.  It follows from Theorem \ref{compact unique} and Theorem \ref{dimension surface} that
${\rm dim}(H/H')\leq 3$. Since $\Gamma$ acts on $K/H'$ minimally and $K/H'$ is of dimension $\geq 1$, $\psi'\psi(\Gamma)$ cannot be finite.
Thus ${\dim}(H/H')\geq 8$ by Theorem \ref{lower dimension}, which is a contradiction.

\end{proof}

\subsection*{Acknowledgements}
I would like to thank Professor Xiaochuan Liu for letting me notice the reference \cite{BK}.


\end{document}